\documentclass[12pt,a4paper]{article}
\usepackage[top=2.5cm,bottom=2.5cm,left=2.5cm,right=2.5cm]{geometry}

\usepackage{amssymb}
\usepackage{amssymb,amsmath,graphicx,color,amsthm}
\usepackage[labelformat=simple]{subcaption}
\usepackage[font=small,labelfont=bf,width=12.5cm]{caption}

\newtheorem{theorem}{Theorem}
\newtheorem{lemma}[theorem]{Lemma}
\newtheorem{corollary}[theorem]{Corollary}
\newtheorem{proposition}[theorem]{Proposition}

\newtheorem{observation}[theorem]{Observation}
\newtheorem{conjecture}[theorem]{Conjecture}
\newtheorem{problem}[theorem]{Problem}
\newtheorem{question}[theorem]{Question}

\newcommand{\set}[1]{\ensuremath{\left\{#1 \right\}}}

\newcommand{\chip}{\ensuremath{\chi_{\rho}}}
\renewcommand{\vec}[1]{\ensuremath{\mathbf{#1}}}

\begin{document}

\title{{\bf Packing coloring of hypercubes\\with extended Hamming codes}}

\author
{
	Petr Gregor\thanks{Charles University, Department of Theoretical Computer Science and Mathematical Logic, Prague, Czech Republic. 
		E-Mail: \texttt{gregor@ktiml.mff.cuni.cz}}, \
	Jaka Kranjc\thanks{General Hospital Novo mesto, Slovenia.
		E-Mail: \texttt{jaka.kranjc@sb-nm.si}}, \ 	
	Borut Lu\v{z}ar\thanks{Faculty of Information Studies in Novo mesto, Slovenia. \newline
		E-Mails: \texttt{borut.luzar@gmail.com}, \texttt{kenny.storgel.research@gmail.com}}  \thanks{Rudolfovo Institute, Novo mesto, Slovenia.}, \
	Kenny \v{S}torgel\footnotemark[3]~\thanks{University of Primorska, FAMNIT, Koper, Slovenia}
}

\maketitle

{
\begin{abstract}
	A {\em packing coloring} of a graph $G$ is a mapping assigning a positive integer (a color) to every vertex of $G$ 
	such that every two vertices of color $k$ are at distance at least $k+1$. 
	The least number of colors needed for a packing coloring of $G$ is called the {\em packing chromatic number} of $G$. 
	In this paper, we continue the study of the packing chromatic number of hypercubes 
	and we improve the upper bounds reported by Torres and Valencia-Pabon 
	({\em P. Torres, M. Valencia-Pabon, The packing chromatic number of hypercubes, Discrete Appl. Math. 190--191 (2015), 127--140})
	by presenting recursive constructions of subsets of distant vertices making use of the properties of the extended Hamming codes.	
	We also answer in negative a question on packing coloring of Cartesian products raised by 
	Bre\v{s}ar, Klav\v{z}ar, and Rall ({\em Problem 5, Bre\v{s}ar et al., On the packing chromatic number of Cartesian products,
	hexagonal lattice, and trees. Discrete Appl. Math. 155 (2007), 2303--2311.}).
\end{abstract}
}


\section{Introduction}

A \textit{packing coloring} of a graph $G$ is a mapping assigning positive integers (colors) to the vertices of $G$ 
such that every two vertices of color $k$ are at distance at least $k+1$. 
The least number of colors needed for a packing coloring of $G$ is called the {\em packing chromatic number} of $G$
and denoted by $\chip(G)$.
Packing coloring was first introduced by Goddard et al.~\cite{GodHedHedHarRal08}
under the name \textit{broadcast coloring} due to its application to solving the broadcast assignment problems. 
Since then it has been studied for various graph families;
see, e.g., a recent survey of Bre\v{s}ar et al.~\cite{BreFerKlaRal20}.

Let us mention two among the most important questions in the area,
which were resolved recently.
The first was whether the packing chromatic number of graphs with maximum degree $3$ is bounded by a constant.
While the answer is affirmative for, e.g., infinite $3$-regular trees and hexagonal lattices,
it was answered in the negative for the general case by Balogh, Kostochka, and Liu~\cite{BalKosLiu18}
with a probabilistic approach. Independently, a constructive evidence that graphs with maximum degree $3$ can have
arbitrarily large packing chromatic number was provided by Bre\v{s}ar and Ferme~\cite{BreFer18}.

The second question was what is the exact packing chromatic number of the infinite square grid.
The lower and upper bounds for its packing chromatic number have been improving 
by different authors over a series of nine papers, 
and finally resolved in 2023 by Subercaseaux and Heule~\cite{SubHeu23},
who proved that the packing chromatic number of the infinite square grid is equal to $15$.

Apart from the abovementioned, 
Cartesian products of graphs are also an intriguing family.
Its standard representatives are hypercubes,
for which, despite their regular structure, 
the exact values of the packing chromatic numbers are not known even 
for the case of the $9$-dimensional cube.
Study of packing coloring of hypercubes was initiated by Goddard et al.~\cite{GodHedHedHarRal08},
who established exact values of the packing chromatic numbers for $Q_n$, with $n \in \set{1,2,3,4,5}$ (see Table~\ref{tbl:bounds}).
The authors of~\cite{GodHedHedHarRal08} also noted that asymptotically,
$$
	\chip(Q_n) \in \Theta\Bigg(\bigg(\frac{1}{2}-O\Big(\frac{1}{n}\Big)\bigg) \cdot 2^n \Bigg)\,.
$$
Torres and Valencia-Pabon~\cite{TorVal15} continued the investigation by 
proving that $\chip(Q_6) = 25$, $\chip(Q_7) = 49$, and $\chip(Q_8) = 95$.
They additionally improved the general upper bound for the packing chromatic number.
\begin{theorem}[Torres and Valencia-Pabon~\cite{TorVal15}]
	\label{thm:torpab}
	For every integer $n \ge 4$, it holds that
	$$
		\chip(Q_n) \le 2^n \bigg( \frac{1}{2} - \frac{1}{2^{\lceil \log_2 n \rceil}} \bigg) - 2 \bigg \lfloor \frac{n-4}{2} \bigg \rfloor + 3\,.
	$$
\end{theorem}

On the other hand, 
currently the best lower bound follows from the following more general result of Bre\v{s}ar, Klav\v{z}ar, and Rall~\cite{BreKlaRal07}.
\begin{theorem}[Bre\v{s}ar, Klav\v{z}ar, and Rall~\cite{BreKlaRal07}]
	\label{thm:breklaral}
	For any two connected graphs $G$ and $H$ of order at least $2$, it holds that
	$$
		\chi_\rho(G \square H) \geq (\chi_\rho(G) + 1)|H| - \mathrm{diam}(G\square H)(|H|-1)-1\,.
	$$
\end{theorem}
As remarked in~\cite{TorVal15}, 
since $Q_n = Q_{n-1} \square K_2$, Theorem~\ref{thm:breklaral} implies that
\begin{equation}
	\label{eq:lower}
	\chip(Q_n) \ge 2\chip(Q_{n-1}) - (n-1)\,.
\end{equation}

In this paper, we improve the upper bound for the packing chromatic number of hypercubes 
by presenting recursive constructions of subsets of distant vertices.
We do that by making use of the properties of the extended Hamming codes in hypercubes 
having dimensions equal to powers of $2$ and their projections to hypercubes of arbitrary dimensions;
our main results are stated in Theorem~\ref{thm:generalUpperBound} and Corollary~\ref{cor:powerof2}.

The paper is structured as follows. In Section~\ref{sec:prel}, 
we introduce notation and main definitions.
In Section~\ref{sec:upper}, we prove the new upper bound for the packing chromatic number of hypercubes,
and in Section~\ref{sec:lower}, we discuss potential improvements of the current lower bound.
Finally, in Section~\ref{sec:con}, we conclude the paper with a short discussion on further work.

\section{Preliminaries}
\label{sec:prel}

We denote the vertex set and the edge set of a graph $G$ by $V(G)$ and $E(G)$, respectively.
The \textit{hypercube} $Q_n$ of dimension $n$, for any integer $n \ge 1$, is the graph with the vertex set 
$V(Q_n)=\mathbb{Z}_2^n$ and two vertices connected if and only if they differ in exactly one coordinate. 
Hence, $|V(Q_n)| = 2^n$ and $|E(Q_n)| = n\cdot 2^{n-1}$. 
We call the vertices with even number of $1$'s {\em even}, otherwise they are {\em odd}.

Generally, elements of $\mathbb{Z}_2^n$ are called {\em binary vectors} (of length $n$).
Note that in this paper, by vector we mean the row vector.
We denote the binary vector with all coordinates equal to $0$ (resp. $1$) by $\vec{0}$ (resp. $\vec{1}$).
A {\em unit vector $e_i$} is a binary vector with the $i$-th coordinate equal to $1$, and all the others equal to $0$.
Usually, the length of a unit vector will be clear from the context;
however, in the cases when we emphasize that its length is $n$, then we denote it by $e_{i,n}$.
The \textit{Hamming distance} between two binary vectors $u$ and $v$, denoted by $d(u,v)$, is the number of coordinates they differ in. 
This is precisely the distance between $u$ and $v$ in $Q_n$. 
The {\em complement} of a binary vector $u$ is the vector $\overline{u}$ such that the Hamming distance between $u$ and $\overline{u}$ is $n$.
We call a pair of vertices $u$ and $\overline{u}$ {\em diametral}.

Whenever we refer to a specific subset of $\mathbb{Z}_2^n$, we use the terminology from coding theory.
In particular, a {\em (binary) code $C$ of length} $n$ is any subset of $\mathbb{Z}_2^n$ and every element of $C$ is called a {\em codeword}. 
The cardinality of the code $C$ is called the {\em size} of $C$. 
The {\em minimum distance} of $C$, denoted by $d(C)$, is the minimum Hamming distance taken over all pairs of distinct codewords of $C$. 
A code $C$ is {\em linear} if a linear combination of any its two codewords $u$ and $v$ 
(i.e., the codeword obtained by summing $u$ and $v$ coordinate-wise) is again a codeword of $C$.
Clearly, every linear code contains the element $\vec{0}$. 


For a binary vector $w$ and a code $C$ (both of length $n$), the set $C + w = \{c + w \,;\, c \in C\}$ is a {\em coset} of the code $C$, 
and $w$ is the {\em coset vector}. 
Note that for every $w \in \mathbb{Z}_2^n$, we have $d(C) = d(C + w)$.

For any positive integer $m$, let $M_m$ be a binary $m \times (2^m-1)$ matrix whose columns 
consist of all nonzero binary vectors of length $m$ in the lexicographic ordering. 
Such a matrix is called a {\em parity check matrix} of the {\em Hamming code $H_m$},
which is defined as the kernel of $M_m$, i.e., $H_m = \set{v \in \mathbb{Z}_2^{2^m-1} \ | \ M_m \cdot v^\top = 0}$ (with the arithmetics over the field $\mathbb{Z}_2$).
Note that $\vec{0} \in H_m$, since $M_m \cdot \vec{0}^\top = \vec{0}$. 
Consequently, Hamming codes are linear codes with minimum distance~$3$, length $n = 2^m - 1$, and size $2^{n-m}$. 
It is known that Hamming codes exist for all $m$~\cite{Ham50},
and this implies that every hypercube $Q_n$, with $n = 2^{m}-1$, 
contains a {\em perfect dominating set}, 
i.e., a set of vertices $H_m$, which are pairwise at distance at least $3$
such that every vertex $v \in V(Q_n) \setminus H_m$ has exactly one neighbor in $H_m$. 
In fact, a Hamming code of length $n = 2^m - 1$ corresponds precisely to a perfect dominating set in $Q_n$. 
Note that, due to symmetry, Hamming codes of length $m \ge 3$ are not unique.

The {\em extended Hamming code $\hat{H}_m$} is obtained from a Hamming code $H_m$ 
by prepending an extra {\em parity bit} to every codeword, 
i.e., to every codeword a coordinate $0$ or $1$ is prepended such that the sum of $1$'s is always even.
The {\em parity check matrix} $\hat{M}_m$ of the extended Hamming code $\hat{H}_m$ is obtained from a parity check matrix of 
the Hamming code $H_m$ by adding a column of $m$ zeroes at the beginning and adding a vector of length $2^m$ composed of ones at the bottom.
The length of $\hat{H}_m$ is hence $n = 2^m$, its size is $2^{n-m-1}$, and its minimum distance is $4$.
Clearly, $\hat{H}_m$ is a linear code. For example, $\hat{H}_2 = \set{0000,1111}$.

Similarly as above, every hypercube $Q_n$, with $n = 2^m$, contains a set of vertices $\hat{H}_m$ 
with pairwise distance at least $4$ such that every odd vertex $v \in V(Q_n) \setminus \hat{H}_m$
has exactly one neighbor in $\hat{H}_m$.

\section{New upper bounds}
\label{sec:upper}

In this section, we present constructions of packing colorings of hypercubes using extended Hamming codes $\hat{H}_m$. 

%
%
%
%

\subsection{Piercing the extended Hamming codes}

Before proving the main theorem, we list several auxiliary results.
We begin by introducing a construction of special codes that will be subsets of $\hat{H}_m$ 
with greater minimum distance.
For every integer $m \ge 2$, we define the code $A_0(m) = \hat{H}_m$.
Furthermore, 
for an integer $i$, with $1 \le i \leq m - 3$, 
the code $A_i(m)$ is defined as
\begin{equation}
	\label{eq:codeAi}
	A_{i}(m) = \set{uu,u\overline{u} \ | \ u \in A_{i-1}(m-1)}\,.\tag{$*$}
\end{equation}
Thus the length of $A_i(m)$ is $2^{m}$ for every $i$.
From the definition, we also infer that for every two integers $i$ and $m$, 
with $m \ge 3$ and $1 \le i \leq m - 3$, we have
	\begin{itemize}
		\item{} $d(A_{i}(m)) = \min \set{2 \cdot d(A_{i-1}(m-1)), 2^{m-1}}$ and
		\item{} $|A_i(m)| = 2|A_{i-1}(m-1)|$.
	\end{itemize}
Moreover, the facts that $d(A_0(m)) = 4$ and $|A_0(m)| = 2^{2^m-m-1}$, 
imply the following.
\begin{observation}
	\label{obs:A_i}	
	For every two integers $i$ and $m$,  
	with $m \ge 3$ and $1 \le i \leq m - 3$, we have	
	\begin{itemize}
		\item[$(i)$] $d(A_{i}(m)) = 2^{i+2}$ and 
		\item[$(ii)$] $|A_{i}(m)| = 2^{2^{m-i}-m+2i-1}$.
	\end{itemize}
\end{observation}

In what follows, we describe additional properties of the sets $A_i(m)$, which we will use to prove the main result.

\begin{lemma}
	\label{lem:subset}
	If $A_{i-1}(m-1) \subseteq \hat H_{m-1}$, then $A_i(m) \subseteq \hat H_m$, for $1 \leq i \leq m - 3$.
\end{lemma}

\begin{proof}
	For any integer $m$, let $\hat{M}_m$ be a parity check matrix of the extended Hamming code $\hat{H}_m$ obtained from the 
	lexicographic parity check matrix of the Hamming code $H_m$, 
	i.e., the parity check matrix in which the column vectors are ordered in the lexicographic ordering.
	
	Let $u \in A_{i-1}(m-1)$. Then, by the definition of $\hat{H}_m$, $\hat{M}_{m-1} \cdot u^\top = \vec{0}$. Since the vectors $u$ and $\overline{u}$ both 
	contain an even number of ones, 
	we have that $\vec{1} \cdot u^\top = 0$ and $\vec{1} \cdot \overline{u}^\top = 0$.
	Furthermore, every row of $\hat{M}_{m-1}$ consists of an even number of ones, hence $\hat{M}_{m-1} \cdot \vec{1}^\top = \vec{0}$.
	
	After dividing the matrix $\hat{M}_m$ to submatrices,
	\[\hat{M}_m = \begin{bmatrix} \vec{0}^\top & v_1^\top & \cdots & v_{2^m-1}^\top  \\
		1 & 1 & \cdots & 1\end{bmatrix} = 
		\begin{bmatrix} 0 & 0 & \cdots & 0 & 1 & 1 & \cdots & 1 \\ 
		\vec{0}^\top & v_1^\top & \cdots & v_{2^{m-1}-1}^\top & \vec{0}^\top & v_1^\top & \cdots & v_{2^{m-1}-1}^\top \\ 
		1 & 1 & \cdots & 1 & 1 & 1 & \cdots & 1 \end{bmatrix},\]
	it is easy to see that
	\[\hat{M}_m \cdot \begin{bmatrix} u^\top \\
	(u+r)^\top \end{bmatrix} = \begin{bmatrix} \vec{0} \cdot u^\top + \vec{1} \cdot (u+r)^\top \\ 
	\hat{M}_{m-1} \cdot u^\top + \hat{M}_{m-1} \cdot (u+r)^\top \end{bmatrix} = \begin{bmatrix} 0 \\ 
	\vec{0}^\top \end{bmatrix},\]
	for $r \in \set{ \vec{0}, \vec{1} }$, so $uu,u\overline{u} \in \hat H_m$.
\end{proof}



\begin{lemma}
	\label{lem:sub2}
	For every three integers $m$, $j$, and $i$, 
	with $m \ge 2$, $j \ge 1$ and $0 \le i \le m-2-j$, it holds that 
	$$	
		A_{i+j}(m) \subseteq A_i(m)\,.
	$$
\end{lemma}

\begin{proof}
	We proceed by induction on $i$. 
	For $i=0$, we have by Lemma~\ref{lem:subset} that $A_{j}(m)\subseteq A_0(m) = \hat{H}_m$ for any $m\ge 2$ and $j\ge 1$. 
	For $i>0$, we have that $m \ge 3$ and 
	$$
		A_{i+j}(m)=\{uu,u\overline{u} \mid u\in A_{i+j-1}(m-1)\} \subseteq \{uu, u\overline{u} \mid u\in A_{i-1}(m-1)\}=A_{i}(m)\,,
	$$
	since $A_{i-1+j}(m-1)\subseteq A_{i-1}(m-1)$ by the induction hypothesis.
\end{proof}

Let $C$ and $D$ be two codes of length $k$ and $\ell$, respectively, with $k < \ell$. 
We say that $C$ is a {\em restricted subset} of $D$, denoted by $C\subseteq_r D$, 
if $C$ is a subset of the code $D'$ obtained from $D$ 
by restricting to some set of $k$ coordinates.
For example, $\set{00,11} \subseteq_r \set{010,101,011}$ if we restrict to the first and the third coordinate.

\begin{lemma}
	\label{lem:subA0}
	Let $m$ and $i$ be two integers such that $m\ge 4$ and $1 \le i \le m-3$.
	Then, \hbox{$A_{i-1}(m-1)\subseteq_r A_0(m)$}.
	In particular, we can always restrict to the first half of the coordinates.
\end{lemma}

\begin{proof}
	Let $m$ and $i$ be two integers satisfying the assumption.
	By Lemma~\ref{lem:subset}, we have that $A_{i-1}(m-1)\subseteq \hat{H}_{m-1}$.
	Furthermore, by the construction of the extended Hamming codes, 
	we have that $\hat{H}_{m-1}\subseteq_r \hat{H}_m$ 
	(by restricting to the first half of coordinates).
	By the definition of $A_0(m)$, 
	we hence infer that $A_{i-1}(m-1) \subseteq \hat{H}_{m-1}\subseteq_r \hat{H}_m = A_0(m)$.
\end{proof}

\subsection{Main theorem}

We now turn our focus to constructing a packing coloring of the $n$-dimensional hypercube for $n \ge 5$.
 
%
%

Let $m$ be the smallest integer such that $n \le 2^m$.
For every $i$ with $0 \le i \le m-3$,
let $\tilde{A}_i(m) \subseteq A_i(m)$ be the largest subset of vertices of $Q_{2^m}$ 
having the same suffix $s_i$ of length $2^m-n$; i.e., the same last $2^m-n$ coordinates.
Since there are $2^{2^m - n}$ distinct suffixes, we have that
\begin{equation}
	\label{eq:AiNar}
	|\tilde{A}_i(m)| \geq |A_i(m)|/2^{2^m - n} = 2^{2^m(2^{-i}-1)-m+2i-1+n}\,.
\end{equation}
Moreover, since we can add to every codeword of $\tilde{A}_i(m)$ a binary vector 
with $0$'s at first $n$ coordinates and the suffix $s_i$ at the last $2^m-n$ coordinates, 
obtaining a code with the same distance,
we may assume that $s_i$ is an all zeroes vector.

Next, let the {\em narrowed code} $\tilde{A}_i(m;n)$ be a set of vertices of $Q_n$
such that $\tilde{A}_i(m;n) \subseteq_r \tilde{A}_i(m)$ with restriction to the first $n$ coordinates of $\tilde{A}_i(m)$.
Note that $d(\tilde{A}_i(m;n)) = d(\tilde{A}_i(m)) \ge d(A_i(m))$, since all the codewords in $d(\tilde{A}_i(m))$ 
have the same suffix.

We also define the {\em expanded code} $\check{A}_i(m-1;n)$ obtained from the code $A_i(m-1)$
by appending to every codeword the suffix of length $n - 2^{m-1}$ with every coordinate equal to $0$.
Note that by Lemma~\ref{lem:subA0} and the definition of $\tilde{A}_i(m;n)$, 
we have that $\check{A}_i(m-1;n) \subseteq \tilde{A}_i(m;n)$.

For example, for $n=5$, we have $m=3$ and
\begin{align*}
	A_0(2)=&\{0000,1111\}, \check A_0(2;5)=\{00000,11110\}\,, \\
	A_0(3)=&\{00000000, 11110000, 11001100, 11000011, \\
			& 10101010, 10100101, 10011001, 10010110, \\
			& 00001111, 00110011, 00111100, 01010101, \\
			& 01011010, 01100110, 01101001, 11111111\},\\
	\tilde{A}_0(3)=&\{00000000,11110000\}, 
	\tilde{A}_0(3;5)=\{00000,11110\}\,.
\end{align*}

Now, we are ready to prove the main result. 
\begin{theorem}
	\label{thm:generalUpperBound}
	Let $n \ge 5$ and let $m$ be the smallest integer such that $n \le 2^m$. 
	Then,
	\begin{align}
		\label{eq:bound}
		\chip(Q_n) \le 2^{n-1} + 2^{m} - n - \sum_{i=0}^{m-3}2^{i+1}\cdot \max \big(|A_i(m)|/2^{2^m - n}, |A_i(m-1)| \big) \,.
	\end{align}
	Moreover, if $n$ is even, then the inequality is strict.
\end{theorem}

\begin{proof}
	We prove the theorem by constructing a packing coloring of $Q_n$ using at most the number of colors specified in~\eqref{eq:bound}.
	First, we color all even vertices of $Q_n$ by color $1$, thus coloring $2^{n-1}$ vertices with one color.

	
	Next, let $\mathcal{C}^{\hat{H}_m}$ denote the set of all cosets 
	of the extended Hamming code $\hat{H}_m$ obtained using the unit vectors; 
	i.e., 
	$$
		\mathcal{C}^{\hat{H}_m} = \set{C_i = \hat{H}_m + e_i \ | \ 1 \le i \le n}\,.
	$$
	Clearly, every $C_i$ contains only odd codewords and $C_i \cap C_j = \emptyset$ 
	for every pair of distinct $i$ and $j$.
	
	Now, as in the definition of narrowed codes, 
	we let $\mathcal{C}_{n}^{\hat{H}_m}$ denote the set of all cosets $C_{i,n}$ 
	obtained from the cosets $C_i$ by restricting to the first $n$ coordinates.	
	In the remainder, for each color $c$ in $\set{2,3,\ldots,n-2}$, 
	we determine a subset of codewords (vertices) from $C_{c,n}$ to be colored with $c$.
	We consider the colors from $\{2,\ldots, 2^{m-1}-1\}$ 
	and the colors from $\{2^{m-1},\ldots, n-2\}$ separately.

	\medskip
	{\bf Case 1}:	
	Let $c$ be a color from $\{2,\ldots,2^{m-1}-1\}$ 
	and let $i$ be the smallest integer such that $c < 2^{i+2}$.
	Note that since $c \le 2^{m-1}-1$, we have that $0 \le i \le m-3$.
	
	By Observation~\ref{obs:A_i} and Lemma~\ref{lem:subset}, 
	we have that $A_i(m) + e_c \subseteq C_c$ 
	and $d(A_i(m) + e_c) = 2^{i+2}$.
	Consequently, we also have that
	$D_c = \tilde{A}_i(m;n) + e_{c,n} \subseteq C_{c,n}$ 
	and $d(D_c) \ge 2^{i+2}$.
	This allows us to color all the vertices in $D_c$ with color $c$ for every $c$.
	
	Similarly, by the definition, we have that
	$E_c = \check{A}_i(m-1;n) + e_{c,n} \subseteq C_{c,n}$ and $d(E_c) \ge 2^{i+2}$,
	meaning that also all the vertices in $E_c$ can be colored with color $c$.
	
	From~\eqref{eq:AiNar} and the definition of the expanded code, we infer that
	\begin{align}
		\label{eq:DcEc}	
		|D_c| \ge |A_i(m)|/2^{2^m - n} \quad \textrm{and} \quad  |E_c| = |A_i(m-1)|\,.
	\end{align}
	Since for the size of $D_c$ we only have an estimation, 
	we choose the set $D_c$ or $E_C$, for which we have a better estimation, to color its vertices with color $c$,	
	and thus the number of vertices colored with $c$ is by~\eqref{eq:DcEc} 
	$$
		\max(|D_c|, |E_c|) \ge \max(|A_i(m)|/2^{2^m - n}, |A_i(m-1)|) \,.
	$$

	Note that for each $i$ there are exactly $2^{i+1}$ colors such that $i$ is the smallest integer with $c < 2^{i+2}$; 
	namely, the colors $c \in \{2^{i+1},\ldots, 2^{i+2}-1\}$.
	The total number of vertices colored by the colors from $\{2,\ldots,2^{m-1}-1\}$ is hence at least
	\begin{align}
		\label{eq:case1}	
		\sum_{i=0}^{m-3}2^{i+1}\cdot \max(|A_i(m)|/2^{2^m - n}, |A_i(m-1)|)\,.
	\end{align}

	\medskip
	{\bf Case 2}: 
	Let $c$ be a color from $\{2^{m-1},\ldots, n-2\}$.
	We will show that for each $c$ there exist two vertices at distance at least $c+1$, which we will color with $c$.
	
	Let $\vec{1}_o$ denote the binary vector $\vec{1}$ if $n$ is even, 
	and the binary vector with the first $n-1$ coordinates equal to $1$ and the $n$-th coordinate equal to $0$ if $n$ is odd.	
	
	By the construction, we have that $\vec{0} \in \tilde{A}_0(m;n)$. 	
	If $\vec{1}_o \in \tilde{A}_0(m;n)$, then the two vertices $\vec{0} + e_{c,n}$ and $\vec{1}_0 + e_{c,n}$ are both in $C_{c,n}$,
	and we can color them with color $c$, for every $c \in \set{2^{m-1},\dots,n-2}$, 
	since they are at distance at least $n-1$.	
	In fact, in the case when $n$ is even, we also have a pair of vertices for color $c = n-1$,
	since the distance between the two vectors is $n$.
	
	So, we may assume that $\vec{1}_o \notin \tilde{A}_0(m;n)$.
	Then, there is a vertex $x \in \tilde{A}_0(m;n)$ such that $d(x,\vec{1}_o) = 2$, 
	otherwise, for every vertex $x' \in \tilde{A}_0(m;n)$, we would have $d(x',\vec{1}_o) \ge 4$,
	and, consequently, $\vec{1}_o \in \tilde{A}_0(m;n)$, a contradiction.
	Therefore $x + e_{c,n} \in C_{c,n}$	and for every color $c \in \set{2^{m-1},\ldots,n-4}$, 
	the pair of vertices $\vec{0} + e_{c,n}$ and $x + e_{c,n}$ is at distance at least $c+1$ 
	so we can color them with $c$ (if $n$ is even, then we do this also for color $n-3$). 
	
	It remains to show that we can find two pairs of vertices 
	to be colored with colors $n-3$ and $n-2$ (if $n$ is even, with colors $n-2$ and $n-1$).
	Recall that every pair of vertices in $Q_n$ has at most two common neighbors.
	Moreover, since for any two distinct vertices $x'$ and $x''$
	with $d(x', \vec{1}_o) = 2$ and $d(x'', \vec{1}_o) = 2$
	we have that $d(x',x'') \le 4$, 
	it follows that vertices at distance $2$ from $\vec{1}_o$ can be in $\tilde{A}_0(m;n)$,
	but only one of them, say $x'$, can be in $\tilde{A}_j(m;n)$, for any $j \ge 1$,
	by Lemma~\ref{lem:sub2} and Observation~\ref{obs:A_i}$(i)$.
	This means that at most two neighbors of $\vec{1}_o$ (the two neighbors that are also the neighbors of $x'$)
	can be colored with a color $c' \ge 4$.
	Two neighbors of $\vec{1}_o$ can also be colored with colors $2$ and $3$, and
	therefore, in the neighborhood of $\vec{1}_o$, there are at most four colored vertices, 
	say with a subset of colors from $\set{2,3,c_1,c_2}$.
	
	For each $c' \in \set{n-3,n-2}$ (if $n$ is even, let $c' \in \set{n-2,n-1}$), we proceed as follows. 
	If $\vec{1}_o + e_{c',n}$ is not already colored,
	then we color the pair $\vec{0} + e_{c',n}$, $\vec{1}_o + e_{c',n}$ with $c'$; 
	this is possible, since the distance between the two vertices is at least $n-1$ (if $n$ is even, at least $n$). 
	Otherwise, let $c'' \in \set{2,3,c_1,c_2}$ be the color of the vertex $\vec{1}_o + e_{c',n}$; i.e., $\vec{1}_o + e_{c',n} = x + e_{c'',n}$.
	Let $C'_{c'',n} \subseteq C_{c'',n}$ be the set of vertices in $C_{c'',n}$ that are already colored
	and uncolor the vertices in $C'_{c'',n}$ (which includes the vertices $\vec{0} + e_{c'',n}$ and $x + e_{c'',n}$).
	Note that $c'' < n-3$ (if $n$ is even, $c'' < n-2$) 
	and that the vertices in $C_{1,n}$ and $C_{n,n}$ are not colored yet.
	In the case when $c' = n-3$ (if $n$ is even, in the case when $c'=n-1$), 
	we color the vertices in $C'_{1,n} = C'_{c'',n} + e_{c'',n} + e_{1,n}$ with $c''$ and color the two non-colored vertices $\vec{0} + e_{c',n}$ and $\vec{1}_o + e_{c',n}$ with $c'$.
	In the case when $c' = n-2$, we color with $c''$ the vertices in $C'_{n,n} = C'_{c'',n} + e_{c'',n} + e_{n,n}$	
	and color the two non-colored vertices $\vec{0} + e_{c',n}$ and $\vec{1}_o + e_{c',n}$ with $c'$. 
	By this, we are done with Case~2.
	
	\bigskip
	Now, we compute the total number of vertices colored by the colors from $\set{1,\ldots, n-2}$.
	There are $2^{n-1}$ vertices colored by color $1$,
	and so, summing together~\eqref{eq:case1} (by Case~1) and $2 \cdot ((n-2) - 2^{m-1} + 1)$ (by Case~2),
	we infer that at least
	$$
		2^{n-1} + \sum_{i=0}^{m-3}2^{i+1}\cdot \max(|A_i(m)|/2^{2^m - n}, |A_i(m-1)|) + 2 \cdot ((n-2) - 2^{m-1} + 1)\,.
	$$
	vertices are colored.

	All the remaining vertices of $Q_n$ must be colored with a unique color (i.e., a color used only on one vertex), 
	except in the case of even $n$, when we also repeat the color $n-1$ once.
	The number of unique colors is thus at most
	$$
		2^n - \big(2^{n-1} + \sum_{i=0}^{m-3}2^{i+1}\cdot \max(|A_i(m)|/2^{2^m - n}, |A_i(m-1)|) + 2 \cdot ((n-2) - 2^{m-1} + 1) \big) \,.
	$$
	
	Finally, summing the number of non-unique colors together with the number of unique colors we infer that
	$$
		\chip(Q_n) \le 2^{n-1} - \sum_{i=0}^{m-3}2^{i+1}\cdot \max(|A_i(m)|/2^{2^m - n}, |A_i(m-1)|) - n + 2^m\,.
	$$
	Note that due to the argument in Case~2, the inequality is strict if $n$ is even.
	This completes the proof.
\end{proof}

Theorem~\ref{thm:generalUpperBound} and Observation~\ref{obs:A_i} infer the following corollary
for an upper bound on the packing chromatic number of hypercubes 
with dimesions equal to powers of $2$.
\begin{corollary}
	\label{cor:powerof2}
	Let $m\ge 3$ and $n = 2^m$. Then,
	$$
		\chip(Q_n)\le 2^{2^m-1} - 1 - \sum_{i=0}^{m-3}2^{2^{m-i}-m+3i}\,.
	$$
\end{corollary}

In Table~\ref{tbl:upper}, we present a comparison of the upper bounds on the packing chromatic number of hypercubes
established in this paper with the upper bounds established in~\cite{TorVal15}.
\begin{table}[htp!]
	\centering
	\begin{tabular}{|r|*3{|c}|}
		\hline
		$n$ & Upper bound by Theorem~\ref{thm:torpab} & Upper bound by Theorem~\ref{thm:generalUpperBound} & Difference \\ \hline
5	&15			&15			&0		\\ \hline
6	&25			&25			&0		\\ \hline
7	&49			&49			&0		\\ \hline
{\bf 8}	&{\bf 95}			&{\bf 95}			&{\bf 0}		\\ \hline
9	&223		&223		&0		\\ \hline
10	&445		&445		&0		\\ \hline
11	&893		&893		&0		\\ \hline
12	&1\! 787		&1\! 787		&0		\\ \hline
13	&3\! 579		&3\! 571		&8		\\ \hline
14	&7\! 161		&7\! 137		&24		\\ \hline
15	&14\! 329		&14\! 273		&56		\\ \hline
{\bf 16}	&{\bf 28\! 663}		&{\bf 28\! 543}		&{\bf 120}	\\ \hline
17	&61\! 431		&61\! 311		&120	\\ \hline
18	&122\! 869		&122\! 749		&120	\\ \hline
19	&245\! 749		&245\! 629		&120	\\ \hline
20	&491\! 507		&491\! 387		&120	\\ \hline
21	&983\! 027		&982\! 907		&120	\\ \hline
22	&1\! 966\! 065	&1\! 965\! 945	&120	\\ \hline
23	&3\! 932\! 145	&3\! 932\! 025	&120	\\ \hline
24	&7\! 864\! 303	&7\! 864\! 183	&120	\\ \hline
25	&15\! 728\! 623	&15\! 728\! 503	&120	\\ \hline
26	&31\! 457\! 261	&31\! 457\! 013	&248	\\ \hline
27	&62\! 914\! 541	&62\! 914\! 037	&504	\\ \hline
28	&125\! 829\! 099	&125\! 828\! 067	&1\! 032	\\ \hline
29	&251\! 658\! 219	&251\! 656\! 131	&2\! 088	\\ \hline
30	&503\! 316\! 457	&503\! 312\! 257	&4\! 200	\\ \hline
31	&1\! 006\! 632\! 937	&1\! 006\! 624\! 513	&8\! 424	\\ \hline
{\bf 32}	&{\bf 2\! 013\! 265\! 895}	&{\bf 2\! 013\! 249\! 023}	&{\bf 16\! 872}	\\ \hline
{\vdots}	&\vdots	&\vdots	& \vdots \\ \hline
{\bf 64}  	&{\bf 8\! 935\! 141\! 660\! 703\! 064\! 007}	&{\bf 8\! 935\! 141\! 660\! 166\! 125\! 567}			&{\bf 536\! 938\! 440}		\\ \hline
	\end{tabular}
	\caption{Comparison of upper bounds on $\chip(Q_n)$ for small $n$. In bold lines the bounds for hypercubes with dimensions of power of $2$ are given.}	
	\label{tbl:upper}
\end{table}

\section{Discussion on lower bounds}
\label{sec:lower}

As already mentioned in the introduction,
the exact value of the packing chromatic number of $Q_9$ is still undetermined.
The current lower and upper bounds for hypercubes $Q_n$, $n \in \set{9,10,11}$,
and the exact values for $n \le 8$ are given in Table~\ref{tbl:bounds}.
Note that the upper bounds for $n \in \set{9,10,11}$ are better
than the ones obtained by Theorems~\ref{thm:torpab} and~\ref{thm:generalUpperBound}, 
since they are obtained by special constructions, 
f\mbox{i}tted exactly to a given $n$~\cite{TorVal15}.

\begin{table}[h]
	\centering
	\begin{tabular}{|r|*3{|c}|}
		\hline
		$n$ 	& 	$\chip(Q_n)$ \\ \hline
		0 	& 	1 				\\ \hline
		1 	& 	2 				\\ \hline
		2 	& 	3 				\\ \hline
		3 	& 	5 				\\ \hline
		4 	& 	7 				\\ \hline
		5 	& 	15 				\\ \hline
		6 	& 	25 				\\ \hline
		7 	& 	49 				\\ \hline
		8 	& 	95 				\\ \hline
		9 	& 	$198 \le \cdot \le 211$		\\ \hline
		10 	& 	$395 \le \cdot \le 421$		\\ \hline
		11 	& 	$794 \le \cdot \le 881$		\\ \hline		
	\end{tabular}
	\caption{The exact values and bounds on the packing chromatic numbers of $Q_n$ for $n \le 11$.}	
	\label{tbl:bounds}
\end{table}

In the remainder of this section, we discuss the approach of coloring one half of the vertices (the even ones)
with color $1$. 
Let us define the coloring formally in a more general setting of bipartite graphs.
A packing coloring \textit{respects bipartition} of a bipartite graph $G$ if all vertices of color $1$ are contained in one of the bipartition sets.
Clearly, all the vertices of the respective bipartition set can be colored with $1$. 
Let us denote the minimum number of colors needed for a packing coloring respecting bipartition by $\chi_\rho^B(G)$.
Then, clearly, we have $\chip(G) \le \chi_\rho^B(G)$.
We conjecture that the two invariants have the same values in the case of hypercubes.
\begin{conjecture}
	\label{conj:oddones}
	For every $n \ge 0$, it holds that 
	$$
		\chi_\rho(Q_n) = \chi_\rho^B(Q_n)\,.
	$$
\end{conjecture}

If packing coloring respecting bipartition can indeed always be extended to an optimal packing coloring of a hypercube,
it would infer tight bounds for the packing chromatic numbers of $Q_9$ and $Q_{10}$.
Namely, the lower bounds for both numbers could be increased to the current upper bounds of $211$ and $421$ colors
by considering the maximum sizes of sets of vertices with even distances.
Let $A(n,d)$ denote the maximum size of the set of vertices in $Q_n$ at mutual distance at least $d$.
For small hypercubes exact values of $A(n,d)$ are known and 
we list some of them in Table~\ref{tbl:codes}~\cite{AgrVarZeg01}.
\begin{table}[h]
	\centering
	\begin{tabular}{|r|*3{|c}|}
		\hline
		$n$ 	& 	$d=4$	& 	$d=6$ 	& 	$d=8$ 	\\ \hline
		9 		& 	$20$	& 	$4$		& 	$2$		\\ \hline
		10 		& 	$40$	& 	$6$		& 	$2$		\\ \hline
		11 		& 	$72$	& 	$12$	& 	$2$		\\ \hline
	\end{tabular}
	\caption{Maximum number of vertices in the hypercubes $Q_n$ at mutual distances at least $d$ (see~\cite{AgrVarZeg01} 
		for a more detailed table).}	
	\label{tbl:codes}
\end{table}
One obtains a lower bound on $\chi_\rho^B(Q_n)$ as follows.
Since one half of vertices is colored by $1$, 
we can color at most $A(n,d)$ vertices with color $c$, for $c \in \set{d-2, d-1}$.
This gives the following numbers of odd vertices for given colors:
\begin{center}
	\begin{tabular}{|c||c*8{|c}||c|}
		\hline
		$n$ / color	& 	$1$		& 	$2$		& 	$3$ 	& 	$4$	& 	$5$	& 	$6$	& 	$7$	& 	$8$	& 	$9$	& \# colored vertices	\\ \hline
		9 			& 	$256$	& 	$20$	& 	$20$	& 	$4$	& 	$4$	& 	$2$	& 	$2$	& 	-	& 	-	&	308	\\ \hline
		10 			& 	$512$	& 	$40$	& 	$40$	& 	$6$	& 	$6$	& 	$2$	& 	$2$	& 	$2$	& 	$2$	&	612	\\ \hline
		11 			& 	$1024$	& 	$72$	& 	$72$	& 	$12$& 	$12$& 	$2$	& 	$2$	& 	$2$	& 	$2$	&	1200	\\ \hline
	\end{tabular}
\end{center}	
In particular, we need additional $204$, $412$, and $848$ colors for the remaining vertices in $Q_9$, $Q_{10}$, and $Q_{11}$, respectively.
Therefore, we need at least $211$, $421$, and $857$ colors for the respective hypercubes, meaning that 
the current upper bounds for $Q_9$ and $Q_{10}$ are exact provided that the packing chromatic number and 
the packing chromatic number respecting bipartition coincide.

\begin{proposition}
	Assuming Conjecture~\ref{conj:oddones}, 
	$\chip(Q_9) = 211$, $\chip(Q_{10}) = 421$, and $\chip(Q_{11}) \ge 857$.
\end{proposition}

\section{Conclusion}
\label{sec:con}


In this paper, we presented a construction of packing colorings of hypercubes
using codes which arose from the extended Hamming codes. 
The coloring thus resulted in using every color at most as many times as the previous colors.
Namely, when creating a packing coloring, it seems natural that the lower colors are used 
more often. But is this really always the case? 
\begin{question}
	Is it true that for any graph $G$, 
	there always exists an optimal packing coloring such that
	the sequence $\set{n_i}_1^{\chip(G)}$,
	where $n_i$ is the number of appearances of color $i$, is non-increasing?
\end{question}

One could also consider a generalization of hypercubes, 
where instead of $K_2$, we deal with Cartesian products of complete graphs $K_q$.
Let $K_q^n$ denote the graph obtained by taking the Cartesian product of $n$ copies of $K_q$ 
for some positive integers $n$ and $q$.
The graph $K_q^n$ is of dimension $n$ with the vertex set $V(K_q^n) = \{0,\ldots,q-1\}^n$ 
and with edges between two vertices if and only if they differ in exactly one coordinate. 
Hence, $K_q^n$ has exactly $q^n$ vertices.
The graph $K_q^n$ can also be defined recursively as $K_q^n = K_q^{n-1} \square K_q$. 
Additionally, $K_q^n$ is a $q$-partite graph with each part of size $q^{n-1}$ and diameter $n$ 
(unless $q = 1$, in which case $K_q^n$ is isomorphic to $K_1$). 

In the case of $n = 2$ (the resulting graph is also known as the rook's graph),
determining the packing chromatic number is fairly easy.
\begin{observation}
	Let $q$ be any positive integer. 
	Then $\chip(K_q^2) = q^2 - q + 1$. 
\end{observation}

\begin{proof}
	If $q = 1$, then $K_q^n$ is isomorphic to $K_1$ for any $n$ and so $\chip(K_q^n) = 1$.
	We may thus assume that $q \ge 2$. 
	Since $K_q^2$ has diameter $2$, all the colors with the exception of color $1$ must be unique. 
	Since $K_q^2$ is a $q$-partite graph with each part of size $q$, 
	we can color one of the parts with color $1$ and the remaining $q^2-q$ vertices must each receive a unique color, thus $\chip(K_q^2) = q^2 - q + 1$.
\end{proof}

Constructing a packing coloring for graphs with $n \ge 3$ will be harder.
Even using the approach presented in this paper seems to be much more complex, 
since one has more vertices in every base graph $K_q$ which may 
receive several distinct colors; 
in our case, one of the vertices was always of color $1$ 
and the other of a color given by a subset of extended Hamming codes or a unique one.
\begin{problem}
	Given positive integers $q \ge 3$ and $n \ge 2$, determine a (non-trivial) upper bound for $\chip(K_q^n)$.
\end{problem}


We conclude the paper by answering a problem of Bre\v{s}ar et al.~\cite[Problem~3]{BreKlaRal07}. 
The authors asked the following.
\begin{question}
	Is it true for arbitrary graphs $G$ and $H$ that
    $$
    	\chi_\rho(G \Box H) \le \max \{ \chi_\rho(G) \cdot |H|, \chi_\rho(H) \cdot |G| \}\,?
    $$
\end{question}

The answer to the problem is negative, by taking, e.g., $G=H=Q_{7}$. 
The packing chromatic number of the hypercube $Q_{14} = Q_7 \Box Q_7$ is at least $6439$;
in particular, using~\cite[Table~1]{AgrVarZeg01},
we deduce that the maximum number of vertices of color $i$, for $1\le i \le 13$, is at most $9958$,
and so there must be at least $6426$ unique vertices.
On the other hand, $\chip(Q_7) = 49$ and $|Q_7| = 128$, yielding $6272$ in total.
This is another example that behavior of the packing chromatic number is much more complex than one would wish for.

\paragraph{Acknowledgement.} 
P.~Gregor acknowledges the financial support by Czech Science Foundation grant GA 22--15272S.
B.~Lu\v{z}ar and K. \v{S}torgel were partially supported by the Slovenian Research Agency Program P1--0383 and the projects J1--3002 and J1--4008.
K.~\v{S}torgel also acknowledges support from the Young Researchers program.

\bibliographystyle{plain}
\bibliography{mainBib}

\end{document}